\documentclass[a4paper,reqno,11pt]{amsart}

\usepackage[english]{babel}
\usepackage{microtype}
\usepackage[latin1]{inputenc}
\usepackage[T1]{fontenc}
\usepackage{graphicx}
\usepackage[usenames]{xcolor}
\usepackage{url}
\usepackage{mathtools}
\usepackage{amssymb}
\usepackage[shortlabels]{enumitem}
\usepackage{csquotes}
\usepackage[font=small,width=0.7\textwidth]{caption}

\newtheorem{theorem}{Theorem}[section]
\newtheorem{proposition}[theorem]{Proposition}

\newtheorem{corollary}[theorem]{Corollary}
\newtheorem{remark}[theorem]{Remark}

\theoremstyle{definition}
\newtheorem{assumption}{Assumption}

\newtheorem*{example}{Example}

\makeatletter
\@addtoreset{equation}{section}

\makeatother

\newcommand{\R}{\mathbb{R}}
\newcommand{\C}{\mathbb{C}}

\newcommand{\Ab}{\mathbf{A}}
\newcommand{\eps}{\varepsilon}

\DeclareMathOperator{\supp}{supp}
\DeclareMathOperator{\Div}{div}
\DeclareMathOperator{\curl}{curl}
\DeclareMathOperator*{\Osc}{Osc}
\DeclareMathOperator{\Hess}{Hess}

\author{Bernard Helffer}

\address[Bernard Helffer]{ Laboratoire de Math\'ematiques Jean Leray, 
Universit\'e de Nantes,  2 rue de la Houssini\`ere, 44322 Nantes, France
and Laboratoire de
Math\'{e}matiques, Universit\'{e} Paris-Sud, France. 
}
\email{bernard.helffer@univ-nantes.fr}

\author{Hynek Kova\v r\'{\i}k}

\address[Hynek Kova\v r\'{\i}k]{DICATAM, Sezione di Matematica, Universit\`a degli studi di Brescia, Italy.}
\email{hynek.kovarik@unibs.it}

\author{Mikael P. Sundqvist}

\address[Mikael P. Sundqvist]{Lund University, Department of Mathematical 
Sciences, Box 118, 221\ 00 Lund, Sweden.}
\email{mikael.persson\_sundqvist@math.lth.se}


\begin{document}

\title[On the ground state energy of the Pauli operator]{%
On the semi-classical analysis of
the ground state energy of
the Dirichlet Pauli operator III:
Magnetic fields that change sign.
}

\begin{abstract}
We consider the semi-classical Dirichlet Pauli operator in bounded connected 
domains in the plane. Rather optimal results have been obtained  in previous papers by  Ekholm--Kova\v{r}\'ik--Portmann and
Helffer--Sundqvist 
for the asymptotics of the ground state energy in the semi-classical limit when the magnetic field has constant sign. In this paper,  we focus  
on the case when the magnetic field changes
sign. We show, in particular, that the
ground state energy of this Pauli operator will be exponentially small as the
semi-classical parameter tends to zero and give lower bounds and upper bounds for  this decay rate.
Concrete examples of magnetic fields changing sign on 
the unit disc are discussed.  Various natural conjectures are disproved and this leaves the 
research of an optimal result in the general case still open.
\end{abstract}

\maketitle

\section{Introduction}

\subsection{The Pauli operator}
Let $\Omega$ be a bounded, open, and connected domain in $\mathbb R^2$, let
$B:\Omega\to \R$ be a bounded magnetic field and $h>0$ a semi-classical parameter.
We are interested in the analysis of the ground state energy $\Lambda^D(h,B,\Omega)$ of  
the Dirichlet realisation of the Pauli operator
\begin{equation} \label{pauli-op}
P(B,h) =  \begin{pmatrix}
P_+(B,h)& 0  \\
0 &   P_-(B,h)
\end{pmatrix},
\end{equation}
in $L^2(\Omega, \C^2)$.
Here, the spin-up component $P_+(B,h)$ and spin-down component $P_-(B,h)$ are defined by
\begin{equation}
P_\pm(B,h):= (hD_{x_1} -A_1)^2 + (hD_{x_2}-A_2)^2 \pm h B(x), 
\end{equation}
$D_{x_j} = -i \partial _{x_j}$ for $j=1,2,$ and the vector potential
$\Ab=(A_1,A_2)$ satisfies
\begin{equation}\label{eq:0}
B(x) = \partial_{x_1} A_2 - \partial_{x_2} A_1, \qquad \forall\ x \in\Omega. 
\end{equation}
The reference to $\Ab$ is not necessary when $\Omega$ is simply connected, in which case it will be omitted, but it could play an important role if the domain is not simply connected. 
In the sequel we will write
\begin{equation} 
\lambda_{\pm}^D(h,\Ab, B,\Omega) = \inf \sigma(P_\pm (B,h)) 
\end{equation} 
and skip the reference to $\Ab$ when not needed. The smallest eigenvalue 
of $P(B,h)$ is given by
\begin{equation} \label{Lambda}
\Lambda^D(h, \Ab, B, \Omega) = \min\{  \lambda_{-}^D(h,\Ab, B,\Omega) , \,   \lambda_{+}^D(h,\Ab, B,\Omega) \}\,.
\end{equation}
The Pauli operator  with Dirichlet boundary conditions is non-negative (this follows from an integration by parts, 
or from the view-point that the Pauli operator is the square of a Dirac operator) 
and, as a consequence, the bottom of the spectrum is non-negative,
\[
\lambda_{\pm}^D(h,\Ab, B,\Omega) \geq 0\,.
\]
Moreover, if $\Gamma$ is defined by $\Gamma u = \bar u\,$, then
\[
P_+(B,h)\Gamma  = \Gamma P_-(-B,h)\,.
\]
It  immediately follows that
\begin{equation} \label{lambda-pm}
\lambda^D_{+}(h,\Ab, B,\Omega)=\lambda^D_{-}(h,-\Ab, -B,\Omega). 
\end{equation}
Hence, to understand the properties of $\Lambda^D(h,\Ab, B,\Omega)$ it suffices
to study $\lambda_{-}^D(h,\Ab, B,\Omega)$, and we will mostly do so.  

The behaviour of $\Lambda^D(h,\Ab, B,\Omega)$  is important  
for the evolution properties of the heat equation $\partial_t u + P(B,h)\, u =0$ 
as it gives the rate of convergence of its solutions to the stationary solution $u=0$. 
In particular $\Lambda^D(h,\Ab, B,\Omega)$ determines the rate of the exponential decay
of the heat semi-group generated by the Pauli operator $P(B,h)$.

Let us now specify the amount of regularity we assume about our domain~$\Omega$ and magnetic field~$B$. 
To do so, we introduce the notation $C^{p,+}$ to mean the H\"older class 
$C^{p,\alpha}$, for some unspecified $\alpha>0\,$.

\begin{assumption}\label{assumptions}
The boundary of $\Omega$ is continuous and piecewise in the H\"older 
class $C^{2,+}$. We allow the boundary to have at most a finite number of convex corners, 
each with aperture less than~$\pi$. The magnetic field $B$ is assumed to be in class 
$C^{0,+}(\overline{\Omega})$.
\end{assumption}

For later reference, we introduce the notation $C^{p,+}_{\mathrm{pw}}$ for
the piecewise $C^{p,+}$ condition from the assumption.  From now on 
we will always work under Assumption \ref{assumptions}.


\subsection{The state of art}
The following result was obtained in~\cite{EKP,HPS1}. 

\begin{theorem}[{\cite{EKP,HPS1}}]
\label{thmEKPNSC}
Let $\Omega$ be a connected domain in $\mathbb R^2$.  
If $B$ does not vanish identically in $\Omega$ there exists $\epsilon >0$ such 
that, for all~$h >0$ and for all~$\Ab$ such that $\curl \Ab=B$,
\begin{equation}\label{ineq1}
 \lambda_{-}^D(h,\Ab, B,\Omega) \geq   \lambda^D( \Omega) \, h^2\,  \exp(-\epsilon/h) \,.
 \end{equation} 
Here, $\lambda^D(\Omega)$ denotes the ground state energy of the Dirichlet 
Laplacian in $\Omega$.
 \end{theorem}
In \cite{EKP}, the statement in the theorem is proved under the assumption that~$\Omega$ is 
simply connected.  The generalization to non-simply connected domains, given in \cite{HPS1}, 
was relatively straightforward using domain monotonicity of the ground state energy in the 
case of the Dirichlet problem,
\[
\lambda_{-}^D(h,\Ab, B,\Omega) \geq  \lambda_{-}^D(h,\widetilde B,\widetilde \Omega) \,,
\]
as soon as extensions of $\Ab$ and $B$ to the simply connected envelope $\widetilde\Omega$ of
$\Omega$ were constructed. 

The proof in~\cite{EKP}  gives a way of computing a lower bound for 
$\epsilon$, by considering the oscillation of the scalar potential $\psi$, i.e.~of any solution of 
$\Delta \psi =B$, and optimising over $\psi$. 
For future reference, we introduce a specific choice of the scalar
potential by letting the function $\psi_0$ be a solution of
\begin{equation}\label{defpsi0}
\Delta \psi_0 = B \text{ in } \Omega\,,\quad \psi_0 =0 \text{ on } \partial\Omega.
\end{equation}
 The regularity conditions   on $\Omega$ and $B$ stated in Assumption~\ref{assumptions} guarantee the local regularity of $\psi_0$ by Schauder estimates:  $\psi_0\in C^{2,+} (\Omega)$ and a control of the regularity at the boundary   (see\footnote{The condition reads: $ 2 -\frac{2}{p} < \frac{\pi}{\omega}$ where $\omega$ is the maximal aperture of the corners.} \cite{D}) which can be deduced from: $\psi_0\in W^{2,p}(\Omega)$ for  
some $p>2$. Hence, by Sobolev's embedding, 
$\psi_0$ belongs to  $C^{1,+} (\overline{\Omega}) \cap  C^{2,+} (\Omega) $.

\smallskip

\noindent With the choice \eqref{defpsi0} proposed for simply connected domains in~\cite{HPS},  
we have

\begin{theorem}[{\cite{HPS}}] \label{th1.3}
Let $\Omega$ be a simply connected domain in $\mathbb R^2$.
If $B >0$ in $\Omega$, and if $\psi_0$ satisfies~\eqref{defpsi0}, then, 
for any $h>0$, 
\[
 \lambda_{-}^D(h, B,\Omega) \geq   \lambda^D(\Omega) \, h^2\,  \exp \bigl(2 \inf_\Omega \psi_0 /h\bigr) \,.
\]
In the semi-classical limit,
\[
\lim_{h\rightarrow 0} h \log  \lambda_{-}^D(h,B,\Omega) = 2 \inf_\Omega \psi_0\,.
\]
\end{theorem}

\noindent Theorem \ref{th1.3} shows that $\lambda_{-}^D(h,B,\Omega)$ is exponentially  small as the semi-classical parameter $h>0$ tends to zero.

In the non-simply connected case, effects from the circulations of the magnetic
potential along different components of the boundary could in principle 
introduce a different constant than $\inf_\Omega\psi_0$ inside the exponential 
function. 
Thus, a modified scalar potential, taking the circulation along the holes into 
account, was used in \cite{HPS1}. It turns out, however, that in the 
semi-classical limit, such effects disappear, and it is again $\psi_0$ that
gives the correct asymptotic.

\begin{theorem}[{\cite{HPS1}}] \label{th1.3nsc}
Assume that $\Omega$ is a connected domain in $\mathbb R^2$, and that $B >0$ in $\Omega$. 
If $\psi_0$ is the solution of \eqref{defpsi0}, 
then, for any $\Ab$ such that $\curl \Ab = B$,  
\[
\lim_{h\rightarrow 0} h \log  \lambda_{-}^D(h,B, \Omega) =  2 \inf_\Omega \psi_0\,.
\]
\end{theorem}

\smallskip

\noindent  If $B$ changes sign, then the lower bound on $\lambda_{-}^D(h,B, \Omega)$ depends on the oscillation of $\psi_0$ in $\Omega$.  More precisely, we have

\begin{theorem}[{\cite{EKP}}]\label{th1.5}
Assume that $\Omega$ is a simply connected domain in $\mathbb R^2$ and
that $\psi_0$ satisfies~\eqref{defpsi0}. Then
\begin{equation} \label{ekp-lowerb}
\lambda_{-}^D(h,\Ab, B,\Omega) \geq   \lambda^D( \Omega) \, h^2\,  \exp  \bigl(- 2 \Osc_\Omega \psi_0/h\bigr) \, ,
\end{equation}
where 
\begin{equation}
 \Osc_\Omega \psi_0 = \sup_{\Omega} \psi_0 -\inf_\Omega \psi_0. 
\end{equation} 
\end{theorem}

This paper will not improve any result concerning lower bounds and will be devoted to improving the existing upper-bounds.
 
\noindent

\section{Main results}
\label{s2}
The aim of this paper is to extend the above mentioned results to Pauli operators with sign changing magnetic fields. 
Given a magnetic field~$B$, we introduce the sets $\Omega_B^+$ and $\Omega_B^-$
as the subsets of $\Omega$ where the magnetic field $B$ is positive and negative,
respectively,
\begin{equation}
\label{eq:Omegaplusminus}
\Omega_B^+=\{ x \in \Omega : B(x) >0\}
\quad\text{and}\quad
\Omega_B^-=\{ x \in \Omega : B(x) <0\}.
\end{equation}

\subsection{The ground state energy of $P_{\pm}(B,h)$} 
It turns out that the scalar potential $\psi_0$, 
the solution of~\eqref{defpsi0}, still plays a main role for the asymptotic
of the bottom of the spectrum of the Pauli operator.
If $B>0$ then, by the maximum principle $\psi_0<0$ in $\Omega$, and similarly, if $B<0$ then
$\psi_0>0$ in $\Omega$. For $B$ with varying sign, it might still be the case 
that $\psi_0$ is of constant sign in $\Omega$, but that will depend on $B$, and
the situation is delicate (we study some examples in the Sections~\ref{s-radial} and~\ref{s6}). 

With \eqref{lambda-pm} in mind we focus again on the eigenvalue $\lambda_{-}^D(h,B, \Omega)$. 
Our first result concerns the case when $\psi_0$ attains negative values in $\Omega$.

\begin{theorem} \label{th1.3ga}
Assume that $\Omega$ is a simply connected domain in $\mathbb R^2$, and that 
$\inf_\Omega \psi_0 < 0$, where $\psi_0$ is satisfying~\eqref{defpsi0}. Then 
\begin{equation} \label{eq:proofof1.4}
\limsup_{h\rightarrow 0} h \log  \lambda_{-}^D(h,B, \Omega) \leq   2 \inf_\Omega \psi_0\,.
\end{equation}
\end{theorem}

\begin{remark}
It is a natural question to discuss, in the case when the magnetic field $B$ changes sign,  under which condition on $B$ the assumption $\inf_\Omega \psi_0 < 0$ holds. 
Some results in this direction but for a discontinuous magnetic field with two values of opposite sign are obtained in~\cite{vdBBu}.
\end{remark}

We recall from~\eqref{eq:Omegaplusminus} the definitions of $\Omega_B^+$ and
$\Omega_B^-$, and will now turn to the case when both of them are non-void.

We assume in addition that $\Gamma:=B^{-1} (0)\subset \overline{\Omega}$ is of class $C^{2,+}$  and that $\Gamma \cap \partial \Omega$ is either empty or, if non empty, that the intersection is a finite set,  avoiding the corner points, with transversal intersection.
 
Under this assumption $\Omega^\pm_B$ satisfies the same condition as $\Omega$ 
from Assumption~\ref{assumptions}, and we will denote by $\hat \psi_0$ the solution of 
\begin{equation}\label{defhatpsi0}
\Delta \hat \psi_0 = B \text{ in }  \Omega_B^+,\quad \hat  \psi_0 =0 \text{ on }\partial\Omega_B^+\,.
\end{equation}
By domain monotonicity, with $ \Omega_B^+  \subset \Omega\,$, 
we can  apply Theorem~\ref{th1.3ga}  with $\Omega$ replaced by $\Omega_B^+$ and get

\begin{corollary}\label{th1.4}
Assume that $\Omega$ is a connected domain in $\mathbb R^2$. 
Assume further that  $\hat\psi_0$ satisfies~\eqref{defhatpsi0} and that $ \Omega_B^+$ is non-empty and simply connected.  
Then
\begin{equation} \label{hps-upperb}
\limsup_{h\rightarrow 0} h \log  \lambda_{-}^D(h,B,\Omega) \leq 2 \inf_{\Omega_B^+} \hat \psi_0\,.
\end{equation}
\end{corollary}

 \begin{remark}\label{remcor} 
 Note that the assumption that $\Omega^\pm_B$ is simply connected is not necessary if we use the results of \cite{HPS1}. An assumption of connectedness is also not necessary. In this case, we choose the
  connected component giving the best result.
  \end{remark}

Now, the main problem is to determine if one of the bounds above, i.e.~\eqref{eq:proofof1.4}
and~\eqref{hps-upperb}, is optimal. 
We have two possibly enlightening statements on this question. 
The first one gives a simple criterion under which the upper bound given in 
Corollary~\ref{th1.4} is not optimal.
 
\begin{theorem}\label{th1.6}
Assume that $\Omega$ and $ \Omega_B^+$ are simply connected domain in $\mathbb R^2$, 
that $\Omega_B^+\neq\varnothing$, and that $\hat\psi_0$ satisfies~\eqref{defhatpsi0}.
If $B^{-1} (0)$ either is a compact $C^{2,+}$ closed curve in $\Omega$ or 
a $C^{2,+}$  line crossing $\partial \Omega$ transversally away from the corners, 
then
\[
\limsup_{h\rightarrow 0} h \log  \lambda_{-}^D(h,B,\Omega) <  2 \inf_{\Omega_B^+} \hat \psi_0\,.
\]
\end{theorem}
Two examples where this condition is satisfied is when $\Omega$ is a disk, 
and the magnetic field is either radial, vanishing on a circle, or affine, vanishing
on a line. We will return to them later.

We mentioned earlier that even though $B$ changes sign, it might happen that
the scalar potential $\psi_0$ does not. 
Our second statement says that in this case we actually have the optimal result.
 
\begin{theorem}\label{th1.7}
Suppose that $\Omega$ is a simply connected domain in $\mathbb R^2$. 
If $\psi_0<0$ in $\Omega$, where $\psi_0$ is the solution of~\eqref{defpsi0}, then
\[
\lim_{h\rightarrow 0} h \log  \lambda_{-}^D(h,B,\Omega) = 2 \inf_\Omega  \psi_0\,.
\]
\end{theorem}

\begin{remark}
Other bounds for excited states have been obtained in the case when $B$ has constant sign in \cite{BTRS}. These authors show that the lower bound in Theorem \ref{th1.3}
is not optimal (in terms of the prefactor) and provide another point of view which is for the moment only applied in the case of a positive magnetic field.
\end{remark}

\subsection{The ground state energy of the Pauli operator}
 As already mentioned,  it follows from~\eqref{lambda-pm} that to 
understand the lowest eigenvalue
of each of the components of $P(B,h)$, it suffices to study one of them, with the
extra price that we must do it both for $B$ and $-B$. 
To discuss the lowest eigenvalue $\Lambda^D(h,B,\Omega)$ of the Pauli operator 
$P(B,h)$, we will compare the eigenvalues for the spin-up and spin-down components. 

If the scalar potential $\psi_0$ does not changes sign in $\Omega$, we can
transfer the earlier results to $\Lambda^D(h,B,\Omega)$.

\begin{theorem}\label{thm-Lambda1} 
Let $\Omega$ be a simply connected domain in $\mathbb R^2$, and let $\psi_0$ be 
given by \eqref{defpsi0}. 
If $\psi_0$ does not change sign in $\Omega$, then 
\begin{equation} \label{lim-Lambda} 
\lim_{h\rightarrow 0} h \log  \Lambda^D(h,B,\Omega) = - 2  \Osc_\Omega \psi_0. 
\end{equation} 
\end{theorem}

\noindent Our final result concerns a case when $\psi_0$  could change sign.  Let
\[
\psi_{\mathrm{min}} :=\inf_\Omega\psi_0, \quad \text{and} \quad \psi_{\mathrm{max}} :=\sup_\Omega\psi_0.
\]

\begin{theorem} \label{thm-closed}
Let $\Omega$ be a simply connected domain in $\mathbb R^2$, and let $\psi_0$ be 
given by \eqref{defpsi0}. Assume that
\[
 \psi_{\mathrm{min}} \, < \, \psi_{\mathrm{max}}. 
\]
Assume further that $\psi_0^{-1} (\psi_{\mathrm{max}})$ contains a $C^{2,+}$ closed curve enclosing a non-empty part of $\psi_0^{-1} (\psi_{\mathrm{min}})$ or that $\psi_0^{-1} (\psi_{\mathrm{min}})$ contains a $C^{2,+}$ closed curve enclosing a non-empty part of $\psi_0^{-1} (\psi_{\mathrm{max}})$. Then 
\begin{equation} \label{lim-Lambda-1} 
\lim_{h\rightarrow 0} h \log  \Lambda^D(h,B,\Omega) = - 2  \Osc \psi_0. 
\end{equation} 
\end{theorem}

\begin{remark} 
A  typical example in which the conditions of Theorem \ref{thm-closed} are satisfied is a radial magnetic field which changes sign on a disc, see Section \ref{s-radial}. 
\end{remark}


\section{Proof of Theorems~\ref{th1.3ga},~\ref{th1.7},~\ref{thm-Lambda1}, and~\ref{thm-closed}}\label{s3}
We assume in this section that $\Omega$ is simply connected.  Using a gauge transform if necessary, we can assume without loss of generality  that the magnetic vector potential
$\Ab$ satisfies
\[
\Div \Ab =0 \text{ in } \Omega \quad \text{ and} \quad  \Ab\cdot \nu =0\text{ on } \partial \Omega\,.
\] 
In this case, the solution  $\psi_0$ of \eqref{defpsi0} satisfies:
$\Ab = \nabla^\perp \psi_0 = (-\partial_{x_2} \psi_0\,,\, \partial_{x_1} \psi_0)$.

\begin{proof}[Proof of Theorem~\ref{th1.3ga}]
 Following \cite{HPS} we choose a test function $u$ in the form 
\[
u = v \exp (-\psi_0/h),
\]
where $v_\eta\in C_0^\infty(\Omega)$. An elementary calculation, c.f.~\cite{EKP}, then shows that
\begin{equation} \label{ekp-ub}
\lambda_{-}^D(h,B, \Omega) \leq \frac{h^2 \int_\Omega e^{-2\psi_0/h}\, |(\partial_{x_1} + i\partial_{x_2}) v|^2\, dx}{\int_\Omega e^{-2\psi_0/h}\, |v|^2\, dx}\, .
\end{equation}
Now let $\eta$ be such that $0 < \eta< - \psi_{\rm min}$ and let $E_\eta\subset\Omega$ be a sufficiently small neighbourhood of $\partial\Omega$ chosen so that 
$\psi_0 > -\frac{\eta}{2}$ in $E_\eta$.
Put $\Omega_\eta= \Omega\setminus E_\eta $  and let $v= v_\eta$ be such that $v_\eta=1$ in $\Omega_\eta$. Then 
\[
\int_\Omega e^{-2\psi_0/h}\, |(\partial_{x_1} + i\partial_{x_2}) v_\eta|^2\, dx \leq C_\eta \, e^{\frac{\eta}{h}}
\]
holds true with some constant  $C_\eta$. To estimate the denominator in \eqref{ekp-ub}, let us denote by $x_0\in \Omega_\eta$ the point where $\inf_\Omega \psi_0$ is achieved. Then there exists a constant $c>0$, independent of $h$, and such that 
\[
\psi_0(x) \leq \psi_{\rm min} + c|x-x_0|^2 \qquad \forall\, x\in\Omega. 
\]
Hence for any $\eps >0$ small enough we have 
\[
\int_\Omega e^{-2\psi_0/h} |v_\eta|^2 \, dx \, \geq \, \int_{\Omega_\eta: |x-x_0| \leq \eps}e^{-2\psi_0/h}  \, dx\, \geq \, \pi \eps^2 e^{-\frac{c \eps^2}{h}}\, e^{-\frac{2\psi_{\rm min}}{h}}\, .
\]
Choosing $\eps=h$ and inserting the above estimates into \eqref{ekp-ub} we obtain 
\[
\limsup_{h\rightarrow 0} h \log  \lambda_{-}^D(h,B, \Omega) \leq 2  \psi_{\rm min} + \eta.
\]
Since $\eta$ can be chosen arbitrarily small, this proves the claim.

\end{proof}

 \begin{remark} The assumptions at the boundary can be weaker when  Theorem~\ref{th1.3ga} is applied to an open set $\omega \subset \Omega$ with less regularity. Suppose that we only have $\psi_\omega \in C^{0,+} (\overline{\omega})$, where $\psi_\omega$ is the solution in $H_0^1(\omega)$ of $\Delta \psi_\omega = B$. Looking at the proof of Theorem~\ref{th1.3ga}  we can construct
   an open set $\hat \omega _\eta $ with $C^{2,+}$-boundary such that $\supp v_\eta \subset \hat \omega_\eta \subset \overline{ \hat \omega_\eta } \subset \omega$ and $\psi_{/ \hat \omega_\eta} \in H^2( \hat \Omega_\eta)$. We can then apply the monotonicity argument between $\Omega$ and $\hat \omega_\eta $.
  \end{remark}
  
\begin{proof}[Proof of Theorem~\ref{th1.7}]
The upper bound follows from Theorem \ref{th1.3ga}. Since $\psi_{\mathrm{max}}=0$, the lower bound is a consequence of Theorem \ref{th1.5}.
\end{proof}

\noindent   Notice that \begin{equation}\label{eq3.1}
\psi_{\mathrm{min}} < \psi_{\mathrm{max}}=0\,.
\end{equation}
holds true whenever $B>0$ in $\Omega$ and that \eqref{eq3.1} implies that $\int_\Omega B(x)\, dx\,   \geq \, 0$.

\begin{proof}[Proof of Theorem~\ref{thm-Lambda1}]
We first note that, since $\Osc_\Omega \psi_0 = \Osc_\Omega(-\psi_0)$, 
Theorem \ref{th1.5} in combination with~\eqref{lambda-pm} gives 
\begin{equation} \label{Lambda-lowerb}
\liminf_{h\rightarrow 0} h \log  \lambda_{\pm}^D(h,B,\Omega)  \, \geq \,  - 2  \Osc_\Omega \psi_0\,.
\end{equation}

If $\psi_0<0$ in $\Omega$, the claim follows from Theorem~\ref{th1.7} and \eqref{Lambda-lowerb}. If $\psi_0>0$ in $\Omega$, then Theorem \ref{th1.7} together with equation \eqref{lambda-pm} implies that 
\begin{equation} 
\lim_{h\rightarrow 0} h \log  \lambda_{+}^D(h,B,\Omega) = 2 \inf_\Omega  (-\psi_0) =  - 2  \Osc_\Omega \psi_0\,. 
\end{equation}
The lower bound \eqref{Lambda-lowerb} then again completes the proof.
\end{proof}

\begin{proof}[Proof of Theorem~\ref{thm-closed}]
Suppose first that  $\psi_0^{-1} (\psi_{\mathrm{max}})$ contains a $C^{2,+}$ closed curve $\gamma_1$ enclosing a non-empty part of $\psi_0^{-1} (\psi_{\mathrm{min}})$. Let $\Omega_1\subset\Omega$ be the region enclosed by $\gamma_1$ and  define on $\Omega_1$ the function $\psi_1= \psi_0-\psi_{\mathrm{max}}$. Then $\Delta \psi_1= B$,  $ \inf_{\Omega_1}  \psi_1<0$, and $\psi_1 = 0$ on $\partial\Omega_1=\gamma_1$. The domain monotonicity and Theorem \ref{th1.3ga} thus imply that
\begin{align*}
\limsup_{h\rightarrow 0}\, h \log \lambda^D_{-} (h, B,\Omega) & \leq \limsup_{h\rightarrow 0}\, h \log \lambda^D_{-} (h, B, \Omega_1) \nonumber \\
& \leq \,  2 \inf_{\Omega_1} \psi_1 = - 2 \Osc_\Omega  \psi_0. 
\end{align*}
On the other hand, if $\psi_0^{-1} (\psi_{\mathrm{min}})$ contains a $C^{2,+}$ closed curve $\gamma_2$ enclosing a non-empty part of $\psi_0^{-1} (\psi_{\mathrm{max}})$, then we denote by $\Omega_2\subset\Omega$ the region enclosed by $\gamma_2$ and define  $\psi_2= -\psi_0+ \psi_{\mathrm{min}}$. Hence $\Delta \psi_2= -B$,  $  \inf_{\Omega_2}  \psi_2<0$, and $\psi_2 = 0$ on $\partial\Omega_2=\gamma_2$.  
In view of equation \eqref{lambda-pm}, Theorem \ref{th1.3ga} and the domain monotonicity we then get 
\begin{align*}
\limsup_{h\rightarrow 0}\, h \log \lambda^D_{+} (h, B,\Omega)  
&\leq \limsup_{h\rightarrow 0}\, h \log \lambda^D_{+} (h, B, \Omega_2) \\
& =\limsup_{h\rightarrow 0}\, h \log \lambda^D_{-} (h, -B, \Omega_2) \\
&\leq \, - 2\inf_{\Omega_2} \psi_2 = - 2 \Osc_\Omega \psi_0\,. 
\end{align*}
In either case, an application of inequality \eqref{Lambda-lowerb} completes the proof.
\end{proof}


\section{Proof of Theorem~\ref{th1.6}}

\subsection{A deformation argument}
We have seen that we can have $\psi < 0$ in $\Omega$ without having to assume $B >0$ and that once this property is satisfied we can obtain an upper bound by restricting to the subset of $\Omega$ where $\psi$ is negative instead. 
Hence a natural idea is to consider the family of subdomains of $\Omega$ defined by
\[
\mathcal F = \{ \omega \subset \Omega, \, \partial\omega \in C_{\mathrm{pw}}^{2,+}: \  \Delta \psi = B \ \text{in} \ \omega\,  \mbox{ and } \, \psi |_{\partial\omega} =0 \ \Rightarrow \ \psi < 0 \ \text{in }\ \omega \} \, .
\] 
The idea behind the proof of Theorem~\ref{th1.6} 
is to show that there exists $\omega\in \mathcal F$ such that
$\Omega_B^+ \subset \omega$ with strict inclusion. More precisely, we have

\begin{proposition}\label{mainprop}
Let $\omega\in \mathcal F$   and let $\psi_\omega$ be the solution of $\Delta \psi = B$ in $\omega$ such that $ \psi_\omega = 0$ on  $\partial\omega$. If $\partial_\nu \psi_\omega >0$ at some regular point $M_\omega$ of $\partial \omega \cap \Omega~$,  then there exists  $\widehat \omega \in \mathcal F$ with $\widehat \omega \supset  \omega$ such that the solution to 
\begin{equation} \label{psitilde}
 \Delta \psi_{\widehat \omega} = B \  \text{in} \  \widehat\omega\,  \quad  \mbox{ and }\quad  \psi_{\widehat \omega} =0 \ \ \text{on} \  \partial\widehat\omega\,, 
\end{equation}
satisfies
\begin{equation}\label{eq:dec}
\inf_{\widehat\omega} \psi_{\widehat \omega} <  \inf_\omega \psi_\omega\,.
\end{equation}
\end{proposition}

\begin{proof}
First we deform $\omega$ into $\widehat \omega$ (smooth and small perturbation). We refer to \cite[Chapter~5]{HP} for different ways to do this. 
We can for example  extend the outward normal vector field to a vector field defined in a  tubular neighbourhood of $\partial \omega$ near $M_\omega$. We call this vector field $X_0$, and take a function $\theta$ in $C_0^\infty(\mathbb R^2)$ with compact support near $M_\omega$ and equal to $1$ near $M_\omega$. 
Considering  the vector field $X:=\theta X_0$, which is naturally defined in $\mathbb R^2$, the associated flow $\Phi_t$ of  $X$ defines the desired deformation for $t$ small. 
We then define $\widehat \omega = \Phi_{t_0} (\omega)$ for some $t_0 >0$. Let $\psi_{\widehat \omega}$ be the solution of \eqref{psitilde}. 
Note that $\psi_\omega$ and $\psi_{\widehat\omega} \circ \Phi_{t_0}^{-1}$ are arbitrarily close, for $t_0$ small enough, in $H^2(\omega)$ and hence in $C^0(\bar \omega)$.
Moreover, from the hypothesis it follows that there exists $\epsilon>0$ such that $\partial(\omega \cap B (M_\omega,\epsilon))$ is $C^{2,+}$ regular. The elliptic regularity  theory \cite{D} thus implies that $\psi_\omega$ and $\psi_{\widehat\omega} \circ \Phi_{t_0}^{-1}$ are arbitrarily close in $W^{2,2+\delta}(\omega \cap B(M_\omega, \epsilon))$ for some $\delta>0$, and therefore, by Sobolev embedding theorems,  also in $C^{1,+}(\bar \omega \cap B(M_\omega, \epsilon))$. This means that $\nabla \psi_\omega$ converges to $\nabla\psi_{\widehat\omega} \circ \Phi_{t_0}^{-1}$ in  $C^0 (\bar \omega \cap B(M_\omega, \epsilon))$ as $t_0\to 0$. Since 
$\partial_\nu \psi_\omega >0$ in $M_\omega$, we infer that $\psi_{\widehat \omega}$ is negative in $(\widehat\omega\setminus \omega) \cap B(M_\omega, \epsilon)$ for $\epsilon$ small enough. Hence  $\psi_{\widehat \omega} - \psi_\omega$ is harmonic in $\omega$, non positive on $\partial \omega$, and strictly negative on 
$ (\partial\omega\setminus \partial\widehat\omega) \cap B(M_\omega, \epsilon)$. 
The maximum principle then gives 
$\psi_{\widehat \omega} < \psi_\omega$ in $\omega$. In particular, we get \eqref{eq:dec}.  We also have $\psi_{\widehat \omega} < 0$ in $\widehat\omega$ for $t_0$ small enough. This shows that $\widehat\omega\in \mathcal F$. 
\end{proof}

\subsection{Proof of Theorem~\ref{th1.6}}
With the deformation argument at hand, we are now ready to give a
\begin{proof}[Proof of Theorem \ref{th1.6}]
This is now an immediate application of Proposition~\ref{mainprop} if we can show that at some point of $B^{-1}(0)$, $\partial_\nu \hat\psi_0 \neq 0$.  However, from the Hopf boundary lemma it follows  that at every regular point of $(\partial\Omega_B^+)\cap\Omega$ we have  $\partial_\nu \hat\psi_0 > 0$. Hence the claim.
\end{proof}


\section{Radial magnetic fields} \label{s-radial}
A typical application of the general theorems concerns radial magnetic fields in $\Omega= D(0,R)$, the disk
of radius $R$ centred at $0$. 
The following result is a combination of the theorems appearing in Section~\ref{s2}.
\begin{theorem}\label{thm-radial}
Assume that $\Omega=D(0,R)$, and that the
magnetic field $B$ is radial and continuous. Then, 
\[
\lim_{h\rightarrow 0} h \log  \Lambda^D(h,B,\Omega) = - 2  \Osc \psi_0. 
\]
\end{theorem}

\begin{proof}
We observe that the solution $\psi_0$ of \eqref{defpsi0} is radial, and 
write $r=\bigl(x_1^2+x_2^2\bigr)^{1/2}$ and $\phi_0(r)=\psi_0(x_1,x_2)$.
 If $\phi_0$  has a constant sign in  $(0,R)$, then Theorem \ref{thm-Lambda1} completes the proof. On 
the other hand, if $\phi_0$ changes sign in $(0,R)$, then 
the claim then follows from Theorem \ref{thm-closed}.
\end{proof}

\begin{remark}
Proposition \ref{mainprop} and  Theorem \ref{thm-radial} suggest that if $\psi_0^{-1} (0,+\infty) \cap \Omega \neq \emptyset\,$, then in order to obtain an optimal upper bound one should find an open set $\omega$  in $\Omega$ for which the solution $\psi_\omega$ of $\Delta \psi = B$ in $\omega$ and $\psi =0$ on $\partial \omega$ satisfies in addition $\partial_\nu \psi =0$ on $\partial \omega \cap \Omega$.
\end{remark}

\begin{example}
As an example of a radial field we consider the function
\[
B_\beta(x_1,x_2) = \beta^2-r^2,\quad   0\leq \beta\leq 1,
\]
in the unit disc $D(0,1)$. An explicit solution of
\[
\Delta \psi_0 = B_\beta\text{ in }D(0,1),\quad \psi_0=0\text{ on }\partial D(0,1),
\]
is radial, and given by
\[ 
\psi_0 (x_1,x_2) = \frac{1}{16} \, (r^2 +1-4\beta^2 ) (1-r^2).
\]
Hence for $\beta\in [0,\frac 12) \cup  [\frac{1}{\sqrt{2}}, 1]$ we can apply Theorem \ref{thm-Lambda1}, while the case $\beta \in [\frac 12,  \frac{1}{\sqrt{2}}]$ is covered by Theorem \ref{thm-radial}. A simple calculation then shows that 
\begin{equation*} 
\lim_{h\rightarrow 0} h \log  \Lambda^D(h,B_\beta,D(0,1)) 
= 
- 2  \Osc \psi_0 
= 
-\begin{cases}
\frac {1}{8} (2\beta^2-1)^2 & \beta \in (0,\frac 12) \\
\frac{1}{2} \beta^4 & \beta \in [\frac 12,  \frac{1}{\sqrt{2}}] \\
 \frac {1}{8} (4 \beta^2-1) & \beta \in ( \frac{1}{\sqrt{2}}, 1).
\end{cases}
\end{equation*} 
\end{example}

\begin{remark}
Since $\Lambda^D(h,B_\beta,D(0,1))=\Lambda^D(h,-B_\beta,D(0,1))$
the example above also covers the magnetic field $-B_\beta$, i.e. the case when the
magnetic field is negative inside the domain delimited by the circle $B_\beta^{-1}(0)$.
\end{remark}


\section{A magnetic field vanishing on a line\\joining two points of the boundary}\label{s6}

\subsection{Preliminaries}
We present a numerical study for the case when the zero-set of the magnetic field, 
$B^{-1}(0)$, is a line joining two points of the boundary. 
We consider again $\Omega=D(0,1)$, the disk of radius $1$, and assume that 
\[
B_\beta(x_1,x_2) = \beta - x_1,\quad -1<\beta<1.
\]
This means that $B_\beta^{-1}(0)$ is given by the line $x_1=\beta$. The solution of 
\[
\Delta \psi_\beta = (\beta-x_1) \mbox{ in } D(0,1)\,,\, \psi_\beta =0 \mbox{ on } \partial D(0,1) \,,
\]
is given by 
\begin{equation}\label{excase1a}
\psi_\beta (x_1,x_2)= \frac 18 (x_1-2\beta) \bigl(1-x_1^2-x_2^2\bigr).
\end{equation}
A straightforward calculation shows that
\begin{equation}\label{excase1c}
\max \psi_\beta=
\begin{cases}
\psi_\beta\Bigl(\frac{2\beta+\sqrt{3+4\beta^2}}{3},0\Bigr) & \text{for }-1<\beta<1/2\,,\\
0 &  \text{for } 1/2 \leq \beta < 1\,,
\end{cases}
\end{equation}
and
\begin{equation}\label{excase1d}
\min \psi_\beta=
\begin{cases}
0 & \text{for } -1 < \beta \leq -1/2\,,\\
\psi_\beta\Bigl(\frac{2\beta-\sqrt{3+4\beta^2}}{3},0\Bigr) & \text{for } -1/2<\beta<1\,.
\end{cases}
\end{equation}
It follows that
\begin{equation}\label{excase1e}
\Osc\psi_\beta=
\begin{cases}
\psi_\beta\Bigl(\frac{2\beta+\sqrt{3+4\beta^2}}{3},0\Bigr) & -1 < \beta\leq -1/2\,,\\
\psi_\beta\Bigl(\frac{2\beta+\sqrt{3+4\beta^2}}{3},0 \Bigr)
-
\psi_\beta\Bigl(\frac{2\beta-\sqrt{3+4\beta^2}}{3},0 \Bigr) &  -1/2< \beta<1/2\,,\\
-\psi_\beta\Bigl(\frac{2\beta-\sqrt{3+4\beta^2}}{3},0 \Bigr) & 1/2\leq\beta< 1\,.
\end{cases}
\end{equation}
 When $\beta \in [\frac 12,1)$ we can apply Theorem \ref{th1.3}  and get the optimal result:
\[
\lim_{h\rightarrow 0} h \log \lambda_-^D(h,B,\Omega) =- 2 \Osc (\psi_\beta)\,.
\]
We now look at the case $\beta \in (-\frac 12, \frac 12)$ and compare what is given by the different methods. To apply Corollary \ref{th1.4}, we introduce 
$\Omega_\beta^+$ as the subset of $D(0,1)$ where $B_\beta$ is 
positive, i.e.  $\Omega_\beta^+=\{(x_1,x_2)\in D(0,1)~|~ x_1<\beta\}$.
The solution $\hat\psi_\beta$ of 
\begin{equation}\label{excase1b}
\Delta \hat \psi_\beta = (\beta - x_1) \text{ in } \Omega_\beta^+,\quad \hat \psi_\beta =0 \text{ on } \partial \Omega_\beta^+ \,,
\end{equation}
is not explicit, except in the case $\beta=0$, where we have $\hat\psi_0=\psi_0$ in $\Omega_0^+$.
The oscillation of $\hat\psi_\beta$ can be calculated numerically (See Figure~\ref{fig:oscdiff}).
\begin{figure}[htp]
\centering
\includegraphics[width=0.6\textwidth]{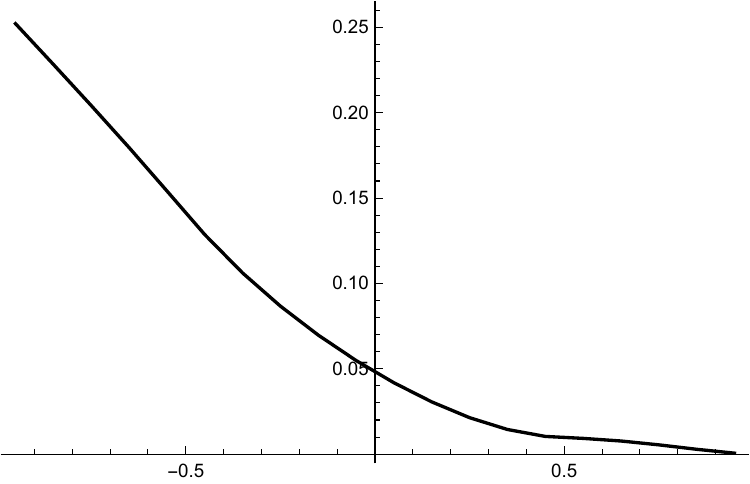}
\caption{Here we have plotted $\Osc \psi_\beta-\Osc \hat\psi_\beta$
as a function of $\beta\in(-1,1)$.}
\label{fig:oscdiff}
\end{figure}
 Hence Corollary \ref{th1.4} gives
\begin{equation}\label{maj1}
\limsup_{h\rightarrow 0} h \log \lambda_-^D(h,B,\Omega)\leq - 2 \Osc (\hat \psi_\beta)\,.
\end{equation}
and Theorem \ref{th1.3ga} gives
\begin{equation}\label{maj2}
\limsup_{h\rightarrow 0} h \log \lambda_-^D(h,B,\Omega)\leq 2 \inf  \psi_\beta
\end{equation}
None of these bounds coincides with the lower bound.

\subsection{Application of Proposition \ref{mainprop}}
We  continue to assume that
\[
-\frac 12 < \beta < \frac 12 \,.
\]
Here we discuss the application  of Proposition~\ref{mainprop} when applied with $\omega = \Omega^+_{2\beta}$.  We note indeed that  $ \Omega^+_{2\beta}$
 belongs to $\mathcal F$.
We need to compute the normal derivative 
\[
 \partial_{x_1} \psi_\beta (x_1,x_2)= \frac 18 (1-x_1^2 - x_2^2) -\frac{x_1}{4} (x_1-2\beta)\,,
\]
with $\psi_\beta$ defined in~\eqref{excase1a}, and to observe that it does not vanish in $\Omega$ on the line $x_1=2\beta$.
As a consequence, Proposition~\ref{mainprop} implies that the upper bound \eqref{maj2} is not optimal. 
Pushing the boundary will indeed improve the upper bound. 
The question of the optimality of the lower bound remains open.

\subsection{Researching a better upper bound}
Given $\beta$ it is interesting to find the largest possible
domain $\Omega_{\mathrm{max}}\subset D(0,1)$ such that the solution to
\begin{equation}
\label{eq:OptimalOmega}
\Delta \psi = \beta-x_1, \text{ in } \Omega_{\mathrm{max}},
\quad
\psi=0, \text{ on } \partial\Omega_{\mathrm{max}}
\end{equation}
is strictly negative in $\Omega_{\mathrm{max}}$. 
The oscillation of that solution could contribute to a candidate
for the optimal constant in the asymptotic of the Pauli eigenvalue. 
We are only able to consider this problem numerically.

To find $\Omega_{\mathrm{max}}$ numerically, we follow (a slightly modified
version of) an iterative procedure
that was kindly suggested by Stephen Luttrell \cite{SL}, described below. We
start by numerically solving the problem on a regular polygon with (many) corners,
positioned on the unit disk.
Then we look at the sign of the solution close to each corner of the polygon, and move the 
corresponding points to make the new domain smaller if the calculated value is 
positive, and larger if it is negative. We also make sure that no point moves 
out from the disk. This gives us a new set of points. We build a new
polygon, and repeat the procedure until the Euclidean distance between the 
corners of two iterations becomes as small as we wish 
(see Figure~\ref{fig:iterativedomains} for an example).

\begin{figure}[htb]
\includegraphics[width=0.8\textwidth]{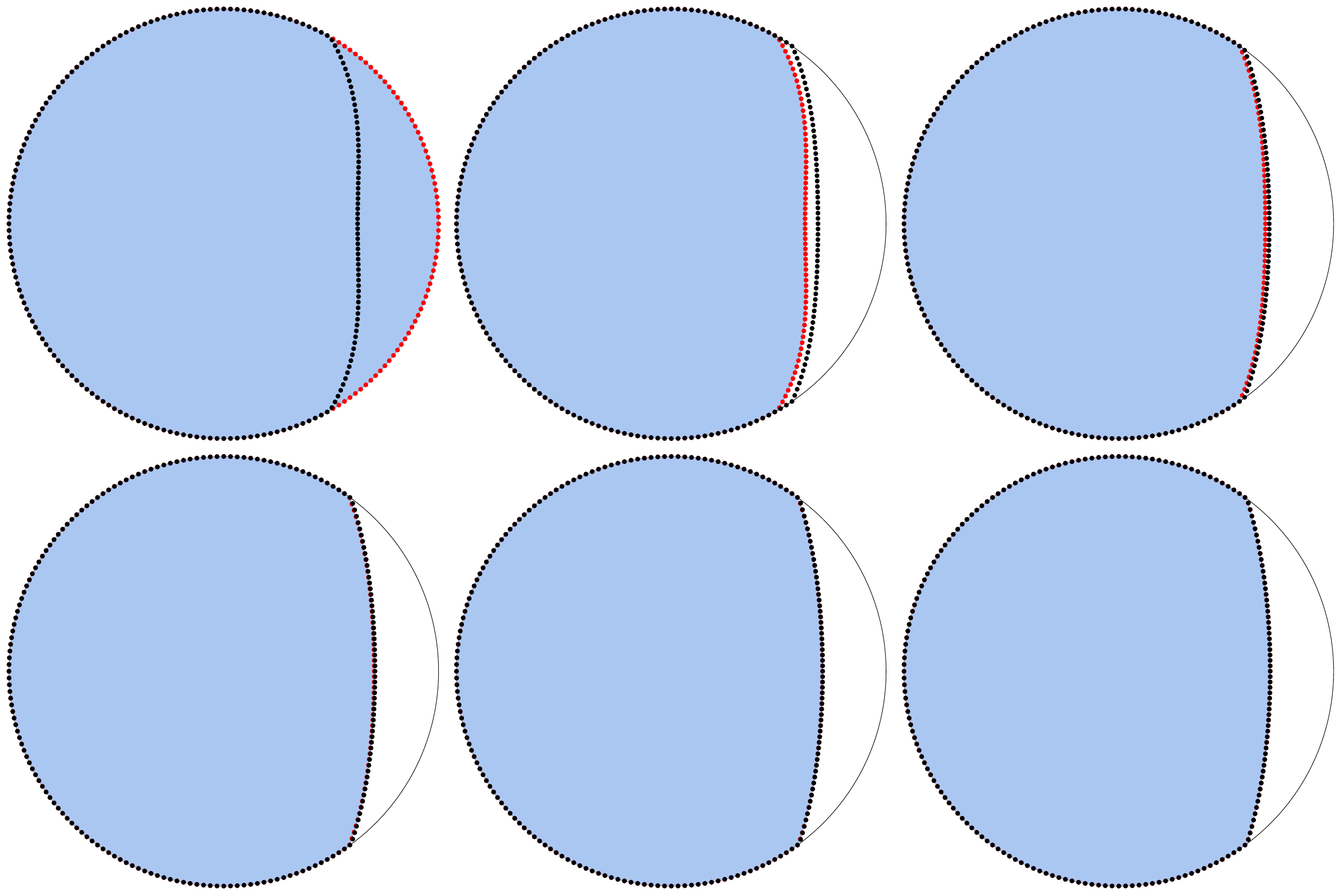}
\caption{The set of domains converge quickly. In this example we start with $200$ vertices, 
$\beta=0.2$, and we exit the loop when the square norm of the difference between 
two consecutive iterations become less than $0.005$ (after five steps). 
The red dots show the vertices used in the domain of the
current step, and the black dots the ones that are calculated for the next step.}
\label{fig:iterativedomains}
\end{figure}

We denote by $\psi_{\beta,\mathrm{opt}}$ the function we end up with after the
iterative procedure (ideally a solution of~\eqref{eq:OptimalOmega}).
In Figure~\ref{fig:psigraphs} we have made a comparison of the oscillation
of $\psi_{\beta,\mathrm{opt}}$ and $\hat\psi_\beta$, and we find that
the oscillation of $\psi_{\beta,\mathrm{opt}}$ is slightly larger than that
of $\hat\psi_\beta$.

\begin{figure}[htbp]
\includegraphics[width=0.6\textwidth]{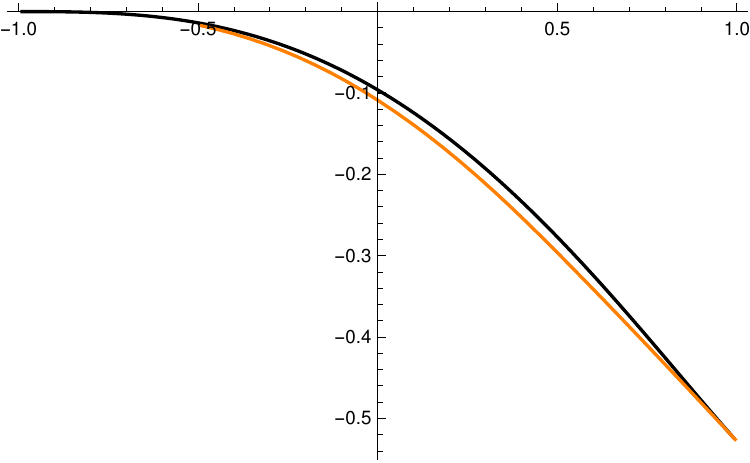}
\caption{The oscillation of $\psi_{\beta,\mathrm{opt}}$ is slightly larger
than that of $\hat\psi_\beta$. Here we show $-2\Osc\psi_{\beta,\mathrm{opt}}$
for $-0.5\leq\beta\leq 1$ 
(orange) and $-2\Osc\hat\psi_\beta$ for $-1\leq\beta\leq 1$ (black).}
\label{fig:psigraphs}
\end{figure}


\section{Coming back to the Witten Laplacian} \label{sec-witten}

\subsection{Former results}
As already mentioned in \cite{HPS}, 
the problem we study is quite closely connected  with the question of analyzing the 
smallest eigenvalue of the Dirichlet realisation of the operator $\Delta^{(0)}_f$ associated in $L^2(\Omega, e^{-2f(x)/h}\, dx)$
with the quadratic form:
\begin{equation}\label{DirForm}
C_0^\infty(\Omega)\ni v 
\mapsto h^2 \,\int_{\Omega}\left|\nabla  v (x)\right|^2\;e^{-2f(x)/h}~dx\;.
\end{equation}
For this case, we can mention  Theorem~7.4 in~\cite[p.~190]{FrWe}, 
which says (in particular) that, if $f$ has a unique critical point which is a non-degenerate local 
minimum $x_{\mathrm{min}}$, then the lowest eigenvalue $\lambda_1(h)$
of the Dirichlet realization $\Delta^{(0)}_{f,h}$ in $\Omega$ satisfies
\begin{equation}\label{resFW} 
\lim_{h\rightarrow 0} - h \,\log \lambda_1(h)
= 2 \inf_{x\in \partial \Omega} \bigl(f(x) - f(x_{\mathrm{min}})\bigr).
\end{equation}

More precise or general results (pre-factors) are given in~\cite{BEGK, BGK,HeNi}. 
This is connected with the semi-classical analysis of Witten 
Laplacians~\cite{W,HKN}.

Starting from
\begin{equation}\label{eq:pr1}
\langle u,P_-u\rangle
= h^2 \int_\Omega \exp \Bigl(- 2 \frac{\psi}{h}\Bigr) \bigl|(\partial_{x_1} + i \partial_{x_2}) v\bigr|^2\, dx,
\end{equation}
with $u = \exp\bigl(- \frac{\psi}{h}\bigr) v$, where $\psi$ satisfies $\Ab = (-\partial_{x_2} \psi, \partial_{x_1} \psi)$,
we observe that if $v$ is real, then 
\[
|(\partial_{x_1} + i \partial_{x_2})v|^2 = |\nabla v|^2.
\]
Using the min-max caracterisation, this implies that the ground state energy of 
the Pauli operator $P_-$ is lower than the the ground state energy
of  
the Dirichlet realisation of the semi-classical Witten Laplacian on $0$-forms:
\[
-h^2 \Delta + |\nabla \psi|^2 - h \Delta \psi\,.
\]
This problem has been analyzed in detail in~\cite{HeNi} 
with computation of pre-factors \emph{but} under generic conditions on 
$\psi$ (see~\cite[Theorem~1.1]{HeNi}) which are not  satisfied in our case. 
The restriction of $\psi$ at the boundary is indeed not a Morse function. 
Hence it is difficult to define the points at the boundary which should be 
considered as saddle points. 
One can  nevertheless think of a small perturbation of $\psi$ to get the 
conditions satisfied (see for example~\cite{HSBott}). Another remark is that, analyzing  the proof in \cite{HeNi}, the Morse assumption at the boundary appears only at the point where the normal exterior derivative of $\psi$ at the
 boundary is strictly positive. Finally, non generic situations are treated in \cite{Mi, GLLN1, GLLN2,Ne}.
One can illustrate this discussion with various instructive examples.
\subsection{First example} 
We discuss what gives the Witten approach in the case when in the unit disk
in~\eqref{excase1a}:
\begin{equation*}
\psi_\beta (x_1,x_2)= \frac 18 (x_1-2\beta) \bigl(1-x_1^2-x_2^2\bigr)\,,
\end{equation*} 
and one should look at the critical points of $\psi_\beta$.
We have
\[
\partial_{x_1} \psi_\beta = \frac 18 \bigl( 1-3 x_1^2 + 4 \beta x_1 - x_2^2\bigr),\quad 
\partial_{x_2} \psi_\beta(x_1,x_2) = \frac 14 x_2 (-x_1 +2\beta).
\]
Hence $\partial_{x_2} \psi_\beta $ vanishes either on the symmetry axis $x_2=0$ or 
on the zero set of $\psi_\beta$. 
The critical points of $\psi_\beta$ are consequently 
either given by $x_2=0$ and $1-3x_1^2 + 4 \beta x_1 =0$,
or by $x_1=2\beta$ and $x_2^2 = 1-4 \beta^2$.

If $\beta \in (-\frac 12, \frac 12)$, we have on $x_2=0$ two critical points 
corresponding to a maximum and a minimum of $\psi_\beta$, and, on $x_1=2\beta$, 
two symmetric critical points corresponding if $\beta\neq 0$  to two saddle points.
 
 If we apply (generalization of) the results of Helffer--Nier \cite{HeNi}, the rate of decay 
will correspond to the difference between the minimum (which is unique) and the 
value at a saddle point (which is in any case $0$).  The Freidlin--Wentzell theorem (see~\eqref{resFW})  does not directly applied (we have  more than one critical point) but does not give a better result. 
We can consequently  not get in this way the oscillation 
of~$\psi_\beta$.
  
Hence  for this example, the upper bound given by the Witten Laplacian does not lead to any improvement in comparison with the result obtained previously by Theorem \ref{th1.3ga}.

\subsection{Second example} 
In the same spirit we now consider in the unit disk
the function
\[
\psi_0 (x_1,x_2) = (1-x_1^2-x_2^2) (6 x_1^2 + 3 x_2^2 - x_2^4 -1)\,.
\]

\begin{figure}[htp]
\centering
\includegraphics[width=0.5\textwidth]{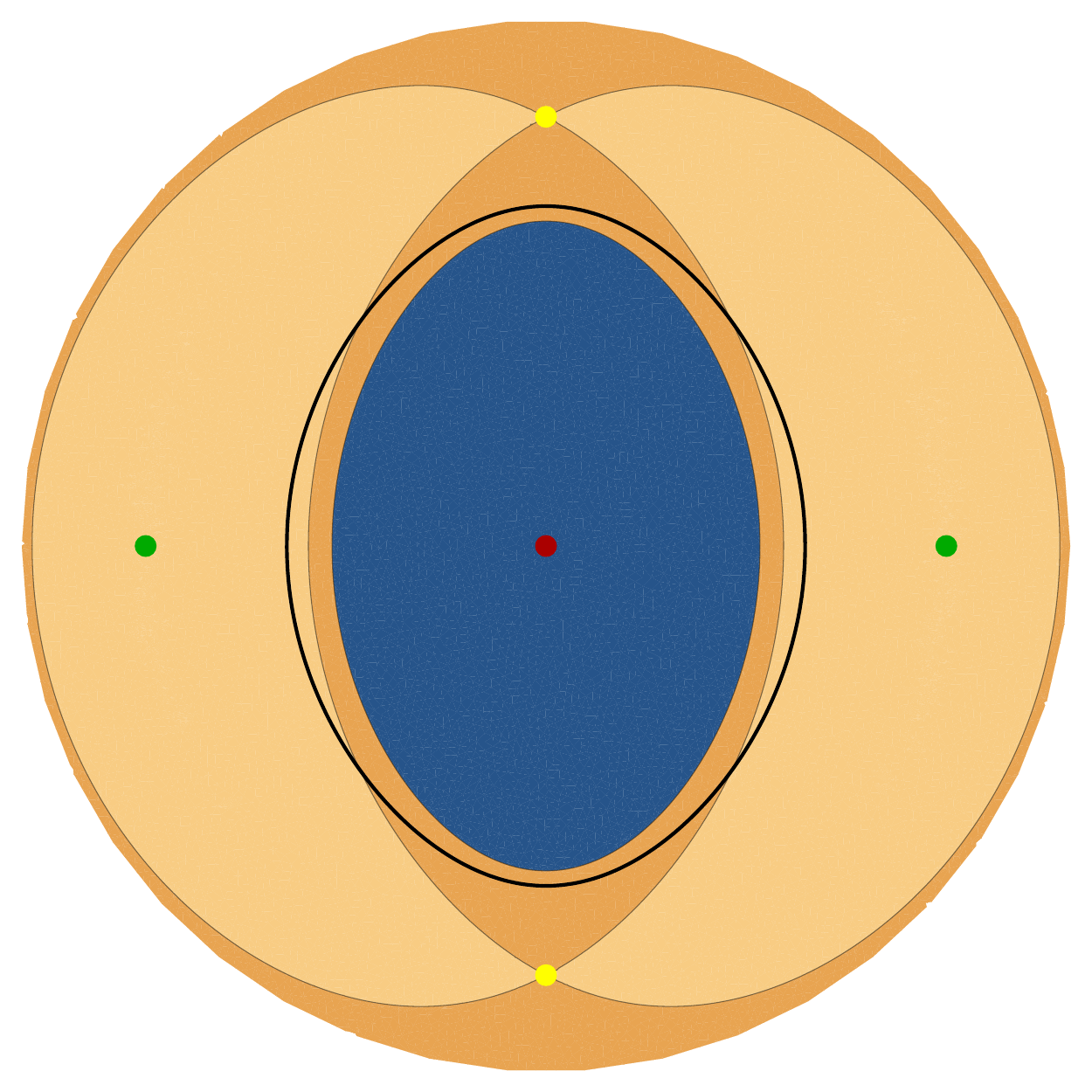}
\caption{Blue means $\psi_0 <0\,$. The  (common)  level set of the yellow saddle points are represented by continuous lines. The two 
maxima are indicated with green dots, the saddle points with yellow, and the
minima in red. The bold black line shows the set $\{\Delta\psi_0=0\}$.}
\label{fig:psicontour}
\end{figure}
We will show below  that this function has one minimum at $(0,0)$, two symmetric maxima on $\{x_2 =0\}$  and two symmetric saddle points on $\{x_1=0\}$.
Moreover, at the boundary
\[
\partial_{x_1}\psi_0 = - 2 x_1(6x_1^2+3x_2^2-x_2^4-1),\quad
\partial_{x_2}  \psi_0 = - 2 x_2(6x_1^2+3x_2^2-x_2^4-1)
\]
so, using  the inequality $(6 x_1^2 + 3 x_2^2 - x_2^4 -1) \geq  1$, we have
\[
\partial_\nu \psi_0:= x_1\partial_{x_1}\psi_0+ x_2 \partial_{x_2} \psi_0 = -2 (6x_1^2+3x_2^2-x_2^4-1) \leq -2\,.
\]
Let us determine the critical points and analyze the Hessian at the critical points. We have
\[
\begin{aligned}
\partial_{x_1} \psi_0 (x_1,x_2)& =2x_1 (7-12 x_1^2-9 x_2^2+ x_2^4) \,, \\
\partial_{x_2} \psi_0 (x_1,x_2)& = 2x_2 (4 -9 x_1^2-8 x_2^2+ 3 x_2^4 + 2 x_1^2x_2^2)\,,\\
\partial^2_{x_1x_1} \psi_0 (x_1,x_2)& = 14 - 72 x_1^2 -18 x_2^2 + 2 x_2^4 \,, \\
\partial^2_{x_1x_2} \psi_0 (x_1,x_2)& =4x_1 x_2(- 9 + 2x_2^2) \,, \\
\partial^2_{x_2x_2} \psi_0 (x_1,x_2)& =  8 -18 x_1^2-48x_2^2+ 30x_2^4 + 12 x_1^2 x_2^2  \,.
\end{aligned}
\]

\paragraph{\bfseries Critical points on the coordinate axes.}
It is easy to compute the critical points living inside the unit disk on $\{x_1=0\}$ and $\{x_2=0\}$.
On $\{x_1=0\}$, we get $ x_{2,\mathrm{sp}}^\pm = \pm\sqrt{ 2/3}$ and one can verify  that this corresponds to  saddle points. We have indeed
\[
\Hess \psi_0 ( 0,  x_{2,sp}^\pm) = 
\begin{pmatrix} 
\frac{26}{9} & 0\\ 
0 & - \frac {32}{3}
\end{pmatrix}
\quad
\text{and}
\quad
\psi_0(0, x_{2,\mathrm{sp}}^\pm) = \frac{5}{27} >0.
\]
On $\{x_2=0\}$, we get $x_{1,\mathrm{min}}=0$ corresponding to a non degenerate 
minimum,
\[
\Hess \psi_0 ( 0,  0) = 
\begin{pmatrix} 14  & 0\\ 0 &  8   
\end{pmatrix}
\quad
\text{and}
\quad
\psi_0(0,0)= -1.
\]
We also get two non degenerate maxima with  $x^\pm_{1,\mathrm{max}} = \pm \sqrt{7/12}$, 
\[
\Hess \psi_0(x^\pm_{1,\mathrm{max}} ,0)  = 
\begin{pmatrix}
-28 & 0\\ 0 & - \frac 52 
\end{pmatrix},
\quad
\text{and}
\quad
\psi_0(x^\pm_{1,\mathrm{max}}, 0) = \frac{25}{24}\,.
\]
We recall that,  at each of these non degenerate  critical points, the computation of  the Hessian and the fact in the case of the extrema that the two eigenvalues are distinct determine the local picture of the integral curves of $\nabla \psi_0$.

\paragraph{\bfseries No critical point outside the coordinate axes.}
Substituting $u=x_1^2$ and $v=x_2^2$ the equations for critical points outside
the coordinate axes becomes
\[
7 -12  u-9 v+ v^2 =0,\quad 4 -9 u -8 v + 3 v^2 + 2 uv =0\
\]
with the additional condition that $u$ and $v$ should belong to the triangle $\{u >0, v>0, u+v < 1\}$.
From the first equation we can solve for $u$, and inserting it into the second
equation we find that $v$ must satisfy the cubic equation
\[
0=\frac{1}{12}\bigl(2v^3+9v^2-v-15\bigr)=\frac{1}{12}(2v+3)(v^2+3v-5).
\]
Clearly, this is not possible if $0<v<1$.


\medskip

\paragraph{\bfseries Conclusions for this example.}
According to the previous remark, we can apply the \enquote{interior results} (see~\cite{HKN}) for getting the main asymptotic. Here we note that Di Gesu--Leli\`evre--Lepeutrec--Nectoux~\cite{GLLN2}  treat the case of saddle points of same value (this case was excluded in \cite{HeNi}).

The magnetic field in the unit disk  is given by
\[
B(x_1,x_2)= 22 - 90 x_1^2 - 66 x_2^2 + 12 x_1^2 x_2^2 + 32 x_2^4\,.
\]
and vanishes at order $1$  along a closed regular curve (see Figure~\ref{fig:psicontour}).
We can now apply either
\begin{enumerate}
\item  Theorems~\ref{th1.3ga} and~\ref{th1.6} in $\Omega^+_B:=B^{-1}{ ((-\infty,0))}$,
\item[] or
\item the following  generalization of Theorem~\ref{thm-closed} which is given below.
\end{enumerate}
\begin{theorem} \label{thm-closedsp}
Let $\Omega$ be a simply connected domain in $\mathbb R^2$, and let $\psi_0$ be 
given by~\eqref{defpsi0}. Assume that there exists a critical point  ${\mathbf x}_{\mathrm{sp}}$
\[
\psi_{\mathrm{min}}=\inf_\Omega  \psi_0  < 0 < \psi_0 ({\mathbf x}_{\mathrm{sp}}):=\psi_{\mathrm{sp}}\,.
\]
Assume further that $\psi_0^{-1} (\psi_{\mathrm{sp}})$ contains a $C^{2,+}$ closed curve enclosing a non-empty part of $\psi_0^{-1} (\psi_{\mathrm{min}})$. Then 
\begin{equation} \label{lim-Lambdasp} 
\lim_{h\rightarrow 0} h \log  \Lambda^D(h,B,\Omega) \leq - 2 (  \psi_{\mathrm{sp}} - \psi_{\mathrm{min}})\,.
\end{equation} 
\end{theorem}
This can be applied in our case. The curve consists  of  two symmetric curves joining the two saddle points. 
But  we can also prove that this is not the optimal result. We can indeed observe that along the boundary of the connected component of $\{x\in\Omega,\ \psi_0(x)<\psi_0 ({\mathbf x}_{\mathrm{sp}})\}$ containing $(0,0)$ the exterior normal derivative of $\psi_0$
is strictly positive (except at the saddle points). Hence Proposition~\ref{mainprop} (with $\psi_\omega = \psi-\psi_{\mathrm{sp}}$)  shows that we can improve the upper bound:
 \begin{equation} \label{lim-Lambdaimp} 
\lim_{h\rightarrow 0} h \log  \Lambda^D(h,B,\Omega) < - 2 (  \psi_{\mathrm{sp}} - \psi_{\mathrm{min}})\,.
\end{equation} 

This example explicitly shows that the Witten Laplacian upper bound is \emph{not} optimal for our example.
\begin{remark} 
 More generic situations can be considered by killing the symmetry by addition of a small perturbation to our \enquote{symmetric example}, by introducing for $\epsilon >0$.
\[
\psi_1(x_1,x_2;\epsilon)= (1-x_1^2-x_2^2) \bigl(6 x_1^2 + 3 x_2^2 - x_2^4 -1+ \epsilon x_1 (x_2-\sqrt{3/2})\bigr).
\]
In this case, this is the saddle point with lowest energy which has to be used for applying Theorem \ref{thm-closedsp}. See Figure~\ref{fig:psiasym}.
\end{remark}

\begin{figure}[htp]
\centering
\includegraphics[width=0.5\textwidth]{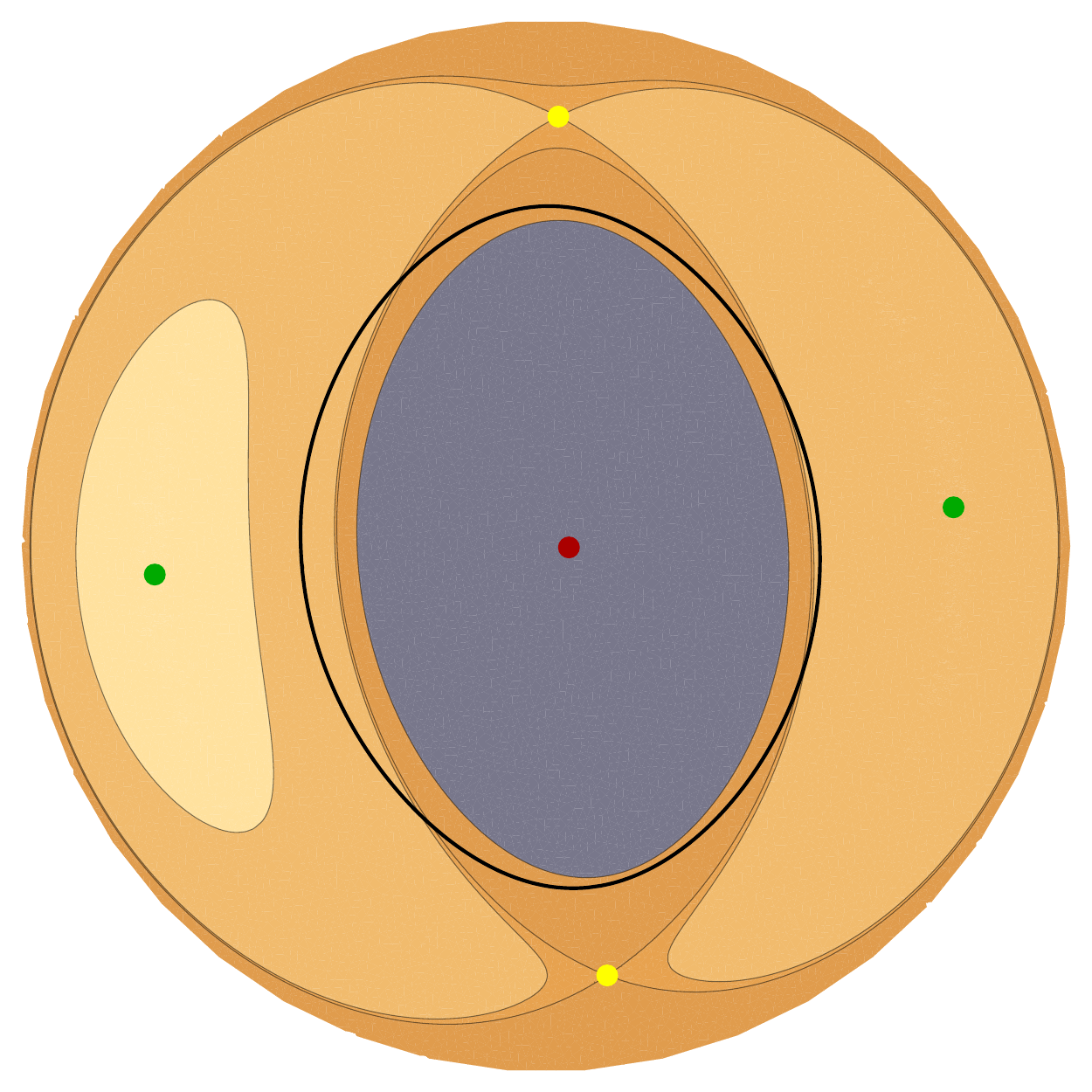}
\caption{An asymmetric example with $\epsilon =1/2$:
The two saddle points (yellow) and the two maxima (green) have no more the same energy. 
The perturbation has been chosen in order to keep one saddle point fixed. The bolder black curve
is the zero set of $\Delta\psi_0$.}
\label{fig:psiasym}
\end{figure}

\begin{remark}
This example could suggest that a candidate to be the optimal set with respect to the \enquote{pushing the boundary} procedure could be the basin of attraction of $\psi_0$ relative to the minimum (see Figure~\ref{fig:basin}). The boundary of this basin consists of four integral curves for the vector field $\nabla \psi_0$ joining the saddle points  to the maxima. We will see in Section~\ref{s8} that this is not in general the right candidate.
\end{remark}

\begin{figure}[htp]
\centering
\includegraphics[width=0.5\textwidth]{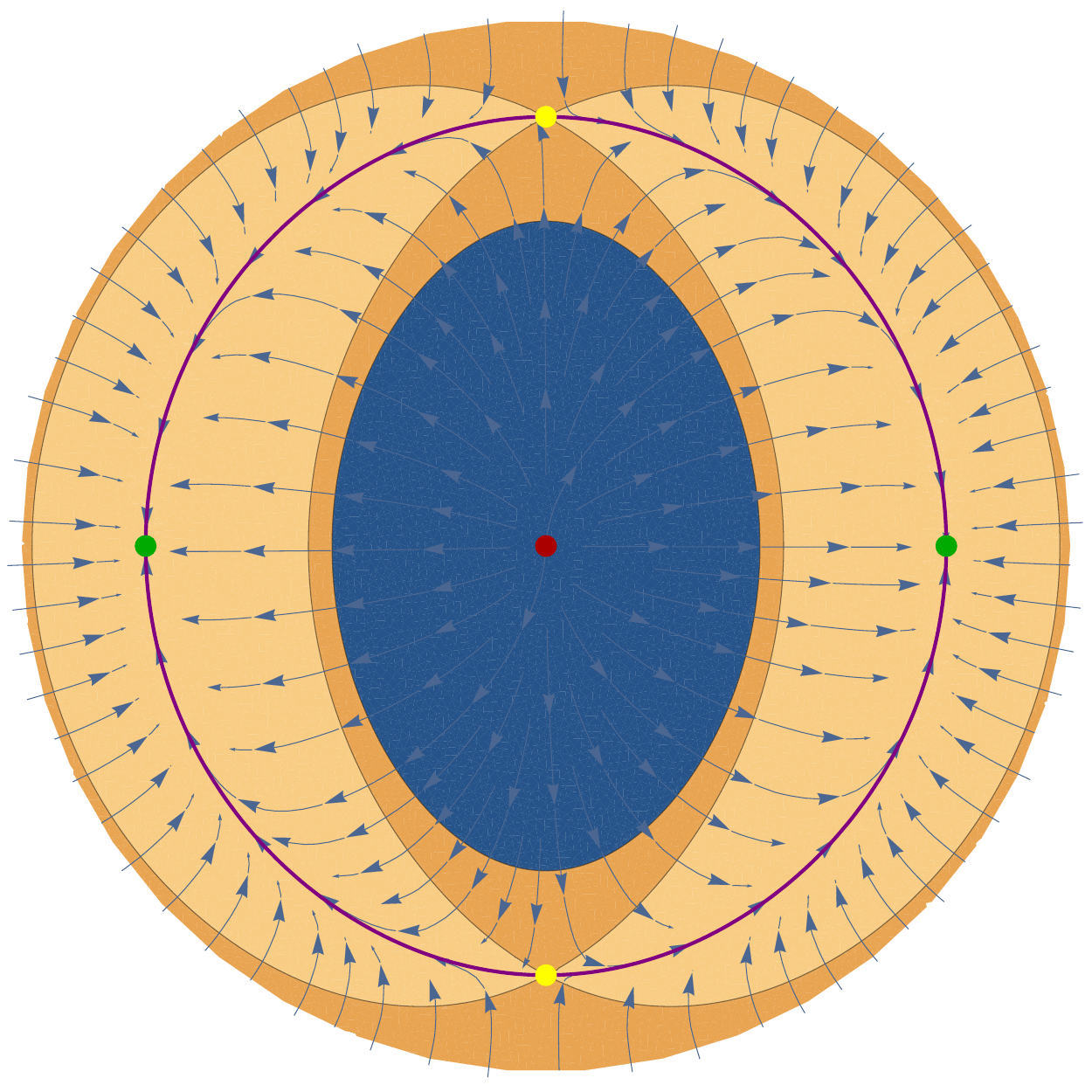}
\caption{The symmetric picture with integral curves of $\nabla \psi_0$. The boundary of the basin of attraction is plotted in purple and is $C^\infty$ except at the maxima.}
\label{fig:basin}
\end{figure}

\section{More examples}\label{s8}

\subsection{A deformation argument revisited}
A natural idea to extend the deformation argument is to consider the family of subdomains of $\Omega$ defined by
\[
\mathcal G = \{ \omega \subset \Omega, \, \partial\omega \in C_{\mathrm{pw}}^{2,+}: \  \Delta \psi = B \ \text{in} \ \omega\,  \mbox{ and } \, \psi |_{\partial\omega} =0 \ \Rightarrow \ \inf_\omega  \psi < 0  \} \, .
\] 
With this new definition, we have

\begin{proposition}\label{mainpropbis}
Let $\omega\in \mathcal G$   and let $\psi_\omega$ be the solution of $\Delta \psi = B$ in $\omega$ such that $ \psi_\omega = 0$ on  $\partial\omega$. If $\partial_\nu \psi_\omega >0$ at some regular point $M_\omega$ of $\partial \omega \cap \Omega$,  there exists $\widehat\omega \supset\omega$ with $\partial\widehat\omega  \in C_{\mathrm{pw}}^{2,+}$  such that the solution to \eqref{psitilde} satisfies
\begin{equation}\label{eq:decbis}
\inf_{\widehat\omega} \psi_{\widehat \omega} <  \inf_\omega \psi_\omega\,.
\end{equation}
Hence $\widehat \omega \in \mathcal G$.
\end{proposition}

\begin{proof} The proof is identical to that of Proposition~\ref{mainprop}. 
\end{proof}

\subsection{An example with one saddle point}
\subsubsection{Critical points and level sets}
Let us consider the function
\[
\psi (x_1,x_2) = - (x_1^2+ x_2^2 - 2 x_1)^2 +  (x_1^2 + x_2^2)\,.
\]
This function does not vanish anymore on boundary of the unit disk, but we can take as open set $\Omega$ the domain delimited by a level set $\psi^{-1}(c)$ for some $c< 0$ to be determined later.
We note that for $c=0$, $\psi^{-1} (0)$ has a simple description in polar coordinates and is described by
\[
r = 2 \cos \theta \pm 1\,.
\]
The function $\psi$ itself becomes
\[
\tilde \psi (r,\theta) = r^2\big[ 1 - (r-2\cos \theta)^2 \big]\,.
\]
We can now look at the critical points.
At $(0,0)$, the Hessian of $\psi$ is given by
\[
\Hess \psi (0,0):= \left(\begin{array}{cc} - 6& 0\\0 & 2\end{array}\right)
\]
Hence $(0,0)$ is a saddle point.
For the other critical points, the figure suggests that the critical points are on $\{x_2=0\}$ (we have $\partial_{x_2} \psi (x_1,0)=0$), see Figure~\ref{fig:psiasym2}. Hence we look at the critical points of 
\[
x_1 \mapsto \psi(x_1,0)= x_1^2 \,[-(x_1-2)^2 +1]= -x_1^2 (x_1-1)(x_1-3)\,
\]
We have
\[
\partial_{x_1} \psi (x_1,0) = - 4 x_1 \Big( x_1^2 - 3 x_1 + \frac 32\Bigr) 
= - 4 x_1 \Bigl(\Bigl( x_1-\frac 32\Bigr)^2  -\frac{3}{4}\Bigr)  \,.
\]
The critical points are on the line $\{x_2=0\}$ and
given by $x_{min} = (\frac{3}{2} - \frac{\sqrt{3}}{2}, 0)$ and $x_{max} = (\frac{3}{2} + \frac{\sqrt{3}}{2}, 0)$. These points are non degenerate extrema. One can also verify that there are no other critical points.
We now explicit the choice of $c$, which should satisfy
\[
- \frac 34 (2 \sqrt{3} -3) = \psi(x_{min}) < c< 0 \,.
\]
Therefore we choose $c=-0.3$. The outer contours in Figures \ref{fig:psiasym2} and \ref{fig:psiasym3} are given by $\psi^{-1}(-0.3)$.

\begin{figure}[htp]
\centering
\includegraphics[width=0.5\textwidth]{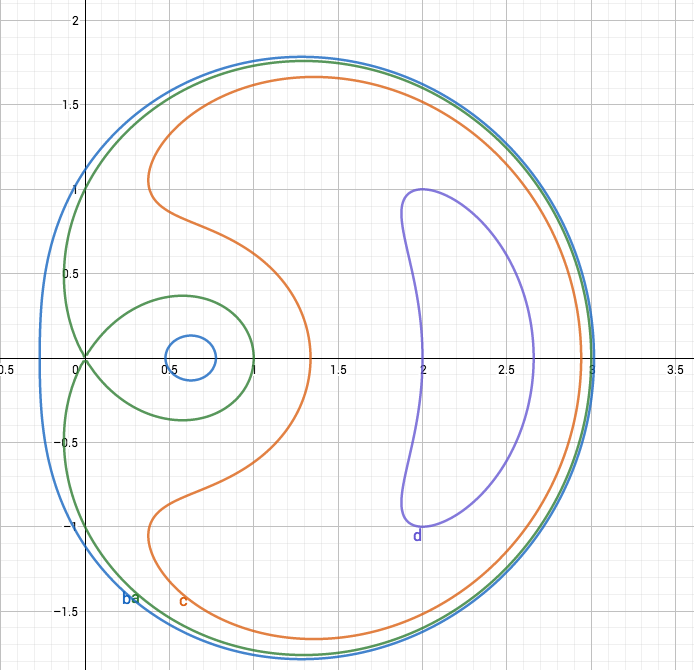}
\caption{$\psi (x_1,x_2)= - (x_1^2+x_2^2-2 x_1)^2+ (x_1^2+x_2^2)$:
The level sets are computed (using Geogebra Dynamics-Mathematics) for $c=- 0.3$ (blue), $c=0$ (green), $c=1$ (orange) and $c=4$ (violet).}
\label{fig:psiasym2}
\end{figure}

We can also compute the corresponding magnetic field  and get
\[
\Delta \psi (x_1,x_2) = 16 \Bigl(-(x_1-1)^2 - x^2_2  + \frac 34\Bigr)\,.
\]
Hence the zero set of the magnetic field is a circle centred at $(1,0)$ of radius $\sqrt{3}/2$.

\begin{figure}[htp]
\centering
\includegraphics[width=0.5\textwidth]{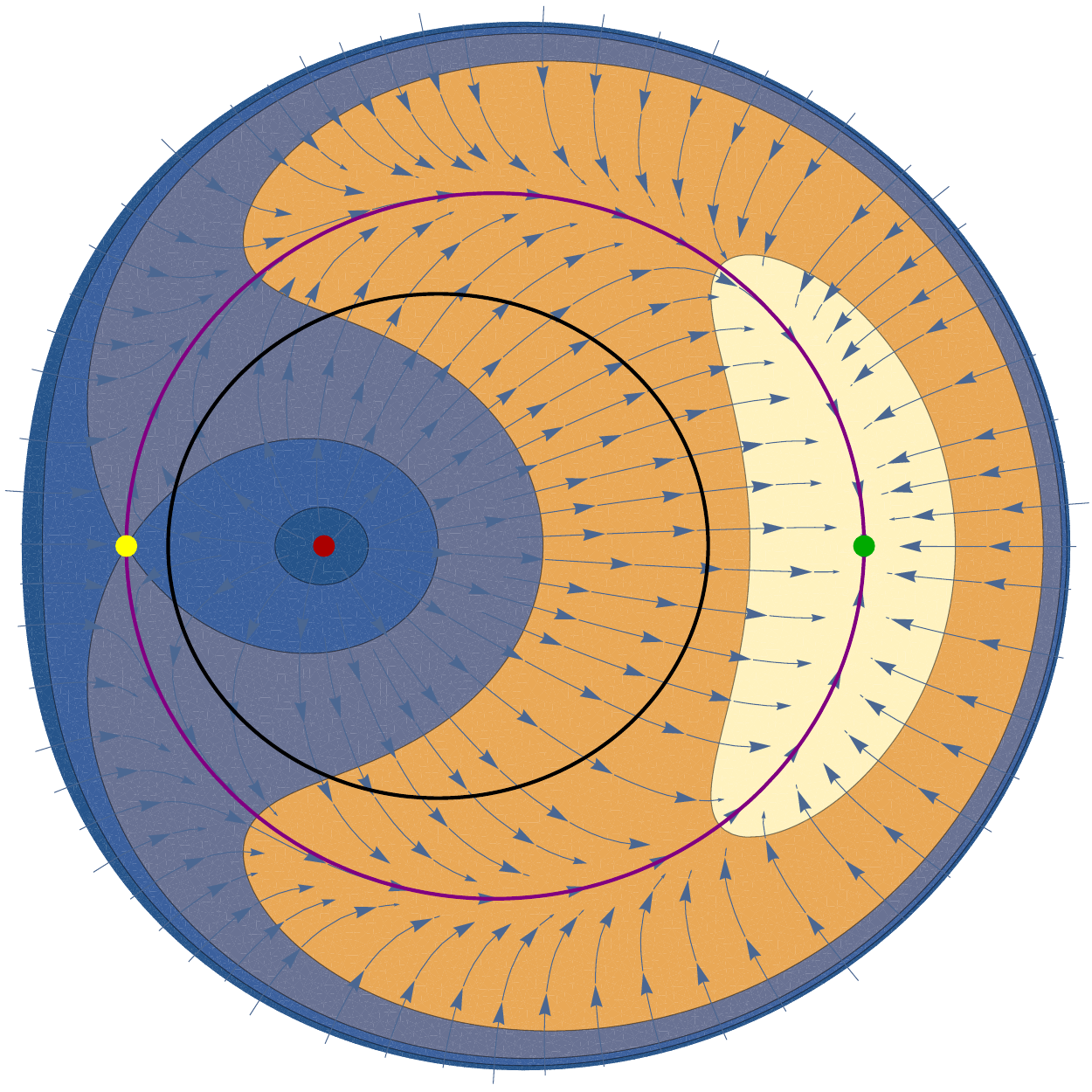}
\caption{  Integral curves of the vector field $\nabla \psi$, where $\psi (x_1,x_2)= - (x_1^2+x_2^2-2 x_1)^2+ (x_1^2+x_2^2)$, and basin of attraction. The orange part and the light yellow parts are regions where the
 function is between 1 and 4 and greater than 4, respectively.}
\label{fig:psiasym3}
\end{figure}

\subsubsection{Morse theory}
When applying the  previously given various criteria, we first consider the connected component $\omega_{ms}$ of the minimum in $\Omega \setminus \psi^{(-1)}(0)$ and the rate of decay is bounded below by $-\psi_{min}$.

To improve the result, we consider the open set 
\[
\omega_{su}:=W^s(x_{min}) \cap W^u(x_{max})\,,
\]
 where $W^s(x_{min})$ is the stable manifold associated with the 
 minimum $x_{min}$ and $W^u(x_{max})$ is the unstable manifold associated with $x_{max}$. Its boundary consists of the two  integral curves relating the saddle point $x_{s} =(0,0)$ to $x_{max}$. We note that along these curves $\psi$ is increasing from $x_{s}$ to $x_{max}$. Note that the curve has vertical tangent at $x_{s}$ and $x_{max}$ and  the boundary of $\omega_{su}$ is $C^\infty$ except possibly at $x_{max}$.  At this point, the analysis of the curve is deduced from the analysis of the Hessian of $\psi$ at the maximum point:
 \[
 \Hess \psi (x_{max})= \left( \begin{array}{cc} - 6 (1+\sqrt{3}) & 0\\ 0 & 2 (1-\sqrt{3}) \end{array}\right)\,.
 \]
 By Sternberg's linearization Theorem \cite{St}, observing that the eigenvalues $\lambda_1=- 6 (1+\sqrt{3}) $ and $\lambda_2= 2 (1-\sqrt{3})$ satisfy the non resonance condition, the regularity
  question about $\partial \omega_{su}$ is sent after conjugation by a local $C^\infty$-diffeomorphism $\Phi$ such that $\Phi (x_{max})=(0,0)$ to the analysis of the linear case.
 The trajectories are given by 
 \[
y_1(t) = \alpha_0  \exp - 6 (1+\sqrt{3})  t \,,\, y_2(t) =  \beta_0  \exp 2 (1-\sqrt{3})  t \,.
 \]
If $\beta_0\neq 0$, we observe that  
 \[
 y_1(t)  = \alpha_0 ( y_2(t)/\beta_0)^\rho\,, 
 \]
 where 
 $
 \rho= 3 (\sqrt{3} +2)\,
  $.
 If $\beta_0 =0$, then the curve is  $y_2=0$.
 Coming back to our boundary of $\omega_{su}$, we see from above that the boundary  $\Phi_\star (\partial \omega_{su})$  is indeed in $C^{11,+}$ (it is locally given by  $y_1 =  \gamma  |y_2|^\rho $  with $\gamma < 0$). \\
 This achieves the proof that $\omega_{su}$ has actually $C^{11,+}$ regularity.\\

We now denote by $\psi_{su}$ the solution of
\[
\Delta \psi_{su} = B (x) \mbox{ in } \omega_{su}\,,\, \psi_{su} =0 \mbox{ on } \partial \omega_{su}\,.
\]
 
 We can then compare $\psi_{su}$ and $\psi$. $\psi_{su} -\psi$ is harmonic in $\omega_{su}$ and by the maximum principle:
\[
 -\psi_{max} < \psi_{su} -\psi < 0 \,,
\]
 which implies
 \begin{equation}\label{strictinequalitya}
- \Osc (\psi) = \psi_{min} -\psi_{max} <  \inf \psi_{su} < \psi_{min}< 0\,.
\end{equation}
 In particular, $- \Osc (\psi)  >  \inf \psi_{su} $.
 Applying now our criterion to the Dirichlet realization of Pauli to $\omega_{su}$, we get

 \begin{equation}\label{strictinequalityb}
\begin{aligned}
\psi_{max} -\psi_{min} & \geq  \frac 12  \limsup_{h\rightarrow 0} \left(- h \log \lambda_-^D(h, B, \Omega)\right) \\
& \geq \frac 12  \liminf_{h\rightarrow 0} \left(- h \log \lambda_-^D(h, B,\Omega)\right)\\
&\geq - \inf \psi_{su} > -  \psi_{min} > c -\psi_{min} \, ,
\end{aligned}
\end{equation}
where the first inequality follows from Theorem \ref{th1.3} applied to $\Omega$ while the third inequality follows from Theorem \ref{th1.3ga} 
and domain monotonicity of the first eigenvalue. 
Hence we have still a window of possible values for $ \frac 12  \lim_{h\rightarrow 0} \left(- h \log \lambda_-^D(h, B)\right) $
  given by the interval $[- \inf \psi_{su} , \Osc \psi]$.
 
\begin{remark} 
Some remarks are in place:
\begin{itemize}[--]
\item  Similar considerations can be done for the other examples. This shows that our considerations are rather generic.
\item We do not know if $\partial_\nu \psi_{su}$ vanishes or has constant sign on $\partial \omega_{su}$. As a consequence, we do not know if we can continue to push the boundary and apply Proposition \ref{mainpropbis}.
\end{itemize}
\end{remark}

\section{Conclusion}
We have initially obtained from the previous papers \cite{EKP,HPS, HPS1} two natural upper bounds and a natural lower bound. We have shown that in general these two initial upper bounds cannot be optimal. 
We have also presented particular cases where the results are optimal. In all these cases, the oscillation of $\psi_0$ is shown to be optimal.

Numerically it could be interesting to see how to \enquote{push the boundary} in Proposition~\ref{mainprop} in order to get a maximal domain.

Finally, as already observed in~\cite{HPS}, one can also expect to get upper bounds by using previous results devoted to the asymptotic of the ground state energy of the Witten Laplacian (see~\cite{BEGK,BGK,FrWe,HKN,HeNi,Mi, W} and the quite recent note of B.~Nectoux \cite{Ne}). This was analyzed in 
Section~\ref{sec-witten} and improved in Section~\ref{s8}.

\section*{Acknowledgements}
The authors would like to thank M.~Dauge for useful discussions about the paper \cite{D}.
B.~H. would like to thank D.~Le~Peutrec for discussions around his work with G. Di Gesu, T. Leli\`evre and B.~Nectoux and N. Raymond for discussions on \cite{BTRS}. H.~K.  has been partially supported by Gruppo Nazionale per Analisi Matematica, la Probabilit\`a e le loro Applicazioni (GNAMPA) of the Istituto Nazionale di Alta Matematica (INdAM). The support of  MIUR-PRIN2010-11 grant for the project ``Calcolo delle variazioni'' (H.~K.), is also gratefully acknowledged.

\bibliographystyle{plain}

\end{document}